\documentclass[12pt]{article}
\usepackage{}

\usepackage{epsfig}
\usepackage{latexsym}
\usepackage{caption}
\usepackage{amsfonts}
\usepackage{amssymb}
\usepackage{mathrsfs}
\usepackage{amsmath}
\usepackage{enumerate}
\usepackage{graphics}
\usepackage{MnSymbol}
\usepackage{float}
\usepackage{pict2e}
\usepackage{subfigure}

\usepackage{color}

\makeatletter

\newcommand{\Rmnum}[1]{\expandafter\@slowromancap\romannumeral #1@}

\makeatother

\begin{document}

\newtheorem{theorem}{Theorem}
\newtheorem{observation}[theorem]{Observation}
\newtheorem{corollary}[theorem]{Corollary}
\newtheorem{algorithm}[theorem]{Algorithm}
\newtheorem{definition}[theorem]{Definition}
\newtheorem{guess}[theorem]{Conjecture}
\newtheorem{claim}[theorem]{Claim}
\newtheorem{problem}[theorem]{Problem}
\newtheorem{question}[theorem]{Question}
\newtheorem{lemma}[theorem]{Lemma}
\newtheorem{proposition}[theorem]{Proposition}
\newtheorem{fact}[theorem]{Fact}

\makeatletter
  \newcommand\figcaption{\def\@captype{figure}\caption}
  \newcommand\tabcaption{\def\@captype{table}\caption}
\makeatother

\newtheorem{acknowledgement}[theorem]{Acknowledgement}

\newtheorem{axiom}[theorem]{Axiom}
\newtheorem{case}[theorem]{Case}
\newtheorem{conclusion}[theorem]{Conclusion}
\newtheorem{condition}[theorem]{Condition}
\newtheorem{conjecture}[theorem]{Conjecture}
\newtheorem{criterion}[theorem]{Criterion}
\newtheorem{example}[theorem]{Example}
\newtheorem{exercise}[theorem]{Exercise}
\newtheorem{notation}{Notation}
\newtheorem{solution}[theorem]{Solution}
\newtheorem{summary}[theorem]{Summary}

\newenvironment{proof}{\noindent {\bf
Proof.}}{\rule{3mm}{3mm}\par\medskip}
\newcommand{\remark}{\medskip\par\noindent {\bf Remark.~~}}
\newcommand{\pp}{{\it p.}}
\newcommand{\de}{\em}
\newcommand{\mad}{\rm mad}
\newcommand{\qf}{Q({\cal F},s)}
\newcommand{\qff}{Q({\cal F}',s)}
\newcommand{\qfff}{Q({\cal F}'',s)}
\newcommand{\f}{{\cal F}}
\newcommand{\ff}{{\cal F}'}
\newcommand{\fff}{{\cal F}''}
\newcommand{\fs}{{\cal F},s}
\newcommand{\s}{\mathcal{S}}
\newcommand{\G}{\Gamma}
\newcommand{\g}{(G_3, L_{f_3})}
\newcommand{\wrt}{with respect to }
\newcommand {\nk}{ Nim$_{\rm{k}} $  }

\newcommand{\q}{\uppercase\expandafter{\romannumeral1}}
\newcommand{\qq}{\uppercase\expandafter{\romannumeral2}}
\newcommand{\qqq}{\uppercase\expandafter{\romannumeral3}}
\newcommand{\qqqq}{\uppercase\expandafter{\romannumeral4}}
\newcommand{\qqqqq}{\uppercase\expandafter{\romannumeral5}}
\newcommand{\qqqqqq}{\uppercase\expandafter{\romannumeral6}}

\newcommand{\qed}{\hfill\rule{0.5em}{0.809em}}

\newcommand{\var}{\vartriangle}

\title{{\large \bf A note on two conjectures that strengthen the four colour theorem}}

\author{  Xuding Zhu\thanks{Department of Mathematics, Zhejiang Normal University,  China.  E-mail: xudingzhu@gmail.com. Grant Number: NSFC 11571319.},}

\maketitle

\begin{abstract}
	
	There are two conjectures concerning planar graph colourings that are strengthenings of the four colour theorem. One concerns signed graph colouring and is proposed by     M\'{a}\v{c}ajov\'{a},   Raspaud and \v{S}koviera. It asserts that every signed planar graph is $4$-colourable. Another concerns list colouring and is proposed by 
 K\"{u}ndgen and Ramamurthi which asserts that if $L$ is a $2$-list assignment of a planar graph $G$, then there is an $L$-colouring   of $G$ such that each colour class induces a bipartite graph. In this note we prove that the first conjecture implies the second one. 
 

\end{abstract}


 The four colour problem is  perhaps the most influential problem  in graph theory. The statement that every planar graph is $4$-colourable remained a  conjecture for over a century before it was confirmed by Appel and Haken \cite{AH} in 1977 by  a computer assisted  proof. Later, a simpler  but still computer assisted proof based on the same general approach was given by Robertson, Sanders, Seymour and Thomas \cite{RST}. The study of the four colour problem generated many powerful tools in graph theory, and also motivated many related theory and challenging problems. In this note, we explore a relation between two conjectures that are strengthenings of the four colour theorem.

 One conjecture concerns colouring of signed graphs. Assume $G$ is a graph. A {\em signature of $G$} is a mapping $\sigma: E(G) \to \{1, -1\}$. A \emph{  signed graph} is a pair $(G, \sigma)$, where $G$ is a graph and $\sigma $ is a signature of $G$.

In the
1980's,
Zaslavsky studied vertex colouring of signed graphs
\cite{Z}. He defined a  colouring of a signed graph $(G, \sigma)$ as a mapping $f: V(G) \to \{\pm k, \pm(k-1) \ldots, \pm 1, 0\}$ such that for any edge $e=xy$ of $G$, $f(x) \ne \sigma_e f(y)$. 
In 2016, M\'{a}\v{c}ajov\'{a},   Raspaud and \v{S}koviera \cite{MRS} defined the chromatic number of a signed graph as follows:   
\begin{definition}
	\label{def-signgcol}
	Assume $(G,\sigma)$ is a signed graph and $k$ is a positive integer.
	If $k=2q$ is even (respectively, $k=2q+1$ is odd), then a $k$-colouring of    $(G,\sigma)$ is a  mapping $f: V(G) \to \{\pm q, \pm(q-1) \ldots, \pm 1 \}$ (respecitively, $f: V(G) \to \{\pm q, \pm(q-1) \ldots, \pm 1,0 \}$ )  such that for any edge $e=xy$ of $G$, $f(x) \ne \sigma_e f(y)$.   The {\em chromatic number} $\chi(G, \sigma)$ of $(G, \sigma)$ is the minimum $k$ such that $(G, \sigma)$ has a  $k$-colouring.
\end{definition}

A signed planar graph is a signed graph $(G, \sigma)$ such that $G$ is a planar graph. As a generalization of the four colour theorem,   M\'{a}\v{c}ajov\'{a},   Raspaud and \v{S}koviera proposed the following conjecture in \cite{MRS}:
 
 \begin{conjecture}
 	\label{conj1}
 	Every signed planar graph is $4$-colourable.
 \end{conjecture}

Another conjecture concerns list colouring of planar graphs. Thomassen proved that every planar graph is $5$-choosable \cite{Tho}. However, as shown by Voigt  \cite{Voigt},    there are planar graphs that are not $4$-choosable. Nevertheless, there is an interesting list version of the four colour theorem, which is a 
 conjecture proposed by  K\"{u}ndgen and Ramamurthi \cite{KR2002}:
 
\begin{conjecture}
	\label{conj2}
	Assume $G$ is a planar graph and $L$ is a $2$-list assignment of $G$. Then there is an $L$-colouring $\phi$ of $G$ such that each colour class    induces a bipartite graph. 
\end{conjecture}

  Conjecture \ref{conj2} is equivalent to say that   if $L$ is a $4$-list assignment of a planar graph $G$ in which colours come in pairs, and for any vertex $v$, a colour occurs in $L(v)$ if its twin occurs in $L(v)$, then $G$ is $L$-colourable. Thus Conjecture \ref{conj2} is also a strengthening  of the four colour theorem.
  
In this note, we prove that Conjecture \ref{conj1} implies Conjecture \ref{conj2}.

\begin{theorem}
	\label{main}
	Assume $G$ is a planar graph. If for any signature $\sigma $ of $G$, the signed graph  $(G, \sigma)$ is $4$-colourable, then for any $2$-list assignment $L$ of $G$, there is an $L$-colouring of $G$ so that each colour class induces a bipartite graph. 
\end{theorem}
\begin{proof}
	Assume  $G$ is a planar graph such that for any signature $\sigma $ of $G$, the signed graph   $(G, \sigma)$ is $4$-colourable. Let $L$ be a $2$-list assignment of $G$. We assume the colours are linearly ordered, say for each vertex $v$ of $G$, $L(v)$ is a set of two positive integers.  We denote by $\min L(v)$ and $\max L(v)$ the minimum and the maximum colour in $L(v)$, respectively.
	We define a signature $\sigma$ of $G$ as follows:
	
	For $e=uv \in E(G)$, let 
	\[
	\sigma(e) =	\begin{cases}
	-1, & \text{ if $\min L(u) = \max L(v)$ or $\min L(v) = \max L(u)$}, \cr
	1, &\text{ otherwise.}
	\end{cases}
	\]
	
	Let $f: V(G) \to \{\pm 1, \pm 2\}$ be a $4$-colouring of $(G, \sigma)$. We define an $L$-colouring $\phi$ of $G$ as follows:
	
		For $v \in V(G)$, let 
		\[
		\phi(v) =	\begin{cases}
		\max L(v), & \text{ if $f(v) \in \{1,2\}$}, \cr
		\min L(v), &\text{ if $f(v) \in \{-1, -2\}$.}
		\end{cases}
		\]
		
		Now we show that for any colour $i$, $\phi^{-1}(i)$ induces a bipartite graph. Let $X_i = \{v \in V(G): \phi(v)=i\}$.
		Let $\psi: X_i \to \{1,2\}$ be defined as 
		$\psi(v) = |f(v)|$. It suffices to show that $\psi$ is a proper colouring of $G[X_i]$. 
		
		Assume $e=uv$ is a positive edge. Then either $i = \min L(u)  = \min L(v)$ or $i = \max L(u)  = \max L(v)$. In the former case, 
		  $f(u), f(v) \in \{-1,-2\}$ and in the latter case,
		  $f(u), f(v) \in \{1,2\}$. Since $e$ is a positive edge, we have $f(u) \ne f(v)$. Therefore $|f(u)| \ne |f(v)|$, i.e., $\psi(u) \ne \psi(v)$.
		  
		Assume $e=uv$ is a negative edge. Then either $i = \min L(u)  = \max L(v)$ or $i = \max L(u)  = \min L(v)$. In the former case, 
		$f(u)  \in \{-1,-2\}$ and $f(v)  \in \{1,2\}$, in the latter case,
		$f(u)  \in \{1,2\}$ and $f(v)  \in \{-1,-2\}$. Since $e$ is a negative edge, we have $f(u) \ne -f(v)$. Hence again $|f(u)| \ne |f(v)|$, i.e., $\psi(u) \ne \psi(v)$.
\end{proof}

\end{document}